\documentclass[11pt]{article}

\usepackage{amsmath,amsthm}

\usepackage{amssymb,latexsym}

\usepackage{enumerate}

\newcommand{\BB}{{\mathcal B}}
\newcommand{\EE}{{\mathcal E}}
\newcommand{\LL}{{\mathcal L}}
\newcommand{\MM}{{\mathcal M}}

\newcommand{\BN}{{\mathbb N}}
\newcommand{\BR}{{\mathbb R}}

\newcommand{\fch}{{\mathbf{1}}}

\newtheorem{theorem}{\bf Theorem}[section]
\newtheorem{proposition}[theorem]{\bf Proposition}
\newtheorem{lemma}[theorem]{\bf Lemma}
\newtheorem{corollary}[theorem]{\bf Corollary}

\theoremstyle{definition}
\newtheorem{definition}[theorem]{Definition}
\newtheorem{example}[theorem]{Example}
\newtheorem{remark}[theorem]{Remark}

\numberwithin{equation}{section}

\begin{document}

\title{On semilinear elliptic equations with diffuse measures}
\author {Tomasz Klimsiak and Andrzej Rozkosz}
\date{}
\maketitle

\begin{abstract}
We consider semilinear equation of the form $-Lu=f(x,u)+\mu$,
where $L$ is the operator corresponding to a transient symmetric regular
Dirichlet form $\EE$, $\mu$ is a diffuse measure with
respect to the capacity associated with $\EE$, and the lower-order
perturbing term $f(x,u)$  satisfies the sign condition in $u$ and  some
weak integrability condition (no growth condition on $f(x,u)$ as a
function of $u$ is imposed). We prove the existence of a solution under mild additional assumptions on $\EE$. We also show that the solution is unique if $f$ is
nonincreasing in $u$.
\end{abstract}

\footnotetext{This work was supported by 
Narodowe Centrum Nauki (Grant No.  2016/23/B/ST1/01543).}

\footnotetext{T. Klimsiak and A. Rozkosz: Faculty of Mathematics and Computer
Science, Nicolaus Copernicus University, Chopina 12/18, 87-100
Toru\'n, Poland. e-mail addresses: tomas@mat.umk.pl (T. Klimsiak), rozkosz@mat.umk.pl (A. Rozkosz).}

\section{Introduction}

Let $E$ be a locally compact separable metric space, $m$ be a
positive Radon  measure on $E$ such that $\mbox{supp\,}[m]=E$, and
let $(\mathcal{E},D(\EE))$ be a regular transient symmetric Dirichlet form on
$L^{2}(E;m)$. In this paper, we consider semilinear equations of the
form
\begin{equation}
\label{eq1.1} -Lu=f(\cdot,u)+\mu.
\end{equation}
In (\ref{eq1.1}), $f:E\times\BR\rightarrow\BR$ is a Carath\'eodory
function satisfying the  so-called ``sign condition":
\begin{equation}
\label{eq1.2} f(x,0)=0,\qquad f(x,y)y\le0,\quad x\in E,\,y\in\BR,
\end{equation}
and $\mu$ is a diffuse measure on $E$ with respect to the capacity
associated with $\EE$, i.e. a bounded signed Borel measure on $E$
which charges no set of capacity zero. As for $L$, we assume that
it is the operator corresponding to $\EE$, i.e. the unique
nonpositive self-adjoint operator on $L^2(E;m)$ such that
\[
D(L)\subset D(\EE),\qquad \EE(u,v)=(-Lu,v),\quad u\in D(L)\,,v\in
D(\EE),
\]
where $(\cdot,\cdot)$ stands for the usual scalar product in $L^2(E;m)$ (see \cite[Section 1.3]{FOT}). Problems of the form (\ref{eq1.1}) with $f$ satisfying the sign conditions are called absorption problems. The model examples of (\ref{eq1.1}) are
\begin{equation}
\label{eq1.3} -\Delta u=f(\cdot, u)+\mu\quad\mbox{in }D,\qquad
u=0\quad\mbox{on }\partial D,
\end{equation}
where  $E:=D$ is a bounded open subset of $\BR^d$ and $\Delta$ is the Laplace operator, and
\begin{equation}
\label{eq1.4} -\Delta^{\alpha/2}u=f(\cdot, u)+\mu\quad\mbox{in
}D,\qquad u=0\quad\mbox{on }\BR^d\setminus D,
\end{equation}
where $\Delta^{\alpha/2}$ is the fractional Laplace operator with $\alpha\in(0,2)$.

The study of problems of the form  (\ref{eq1.3}) with $\mu\in
L^1(D;dx)$ was  initiated by Brezis and Strauss \cite{BS}
(in fact, in \cite{BS} more general second-order elliptic differential operator is
considered). In \cite{BS} it is proved that if $f$ satisfies the sign condition
and
\begin{equation}
\label{eq1.5}
\forall a>0\quad F_a\in L^1(E;m),\quad \mbox{
where }F_a(x)=\sup_{|y|\le a} |f(x,y)|,\,x\in E
\end{equation}
with $E=D$, then there exists a solution to (\ref{eq1.3}) for  $\mu$ belonging to some class which is ``arbitrarily smaller" than $L^1(D;dx)$. If $f$ satisfies stronger monotonicity condition:
\begin{equation}
\label{eq1.6}
(f(x,y_1)-f(x,y_2))(y_1-y_2)\le0,\quad x\in E,\,y_1,y_2\in\BR,
\end{equation}
then  the  solution exists for any  $\mu\in L^1(D;dx)$ and is unique. Later, Gallou\"et and Morel \cite{GM} proved the existence of a solution to (\ref{eq1.3}) for any $\mu\in L^1(D;dx)$ and $f$ satisfying (\ref{eq1.2}), (\ref{eq1.5}). Orsina and Ponce \cite{OP}  have subsequently generalized and strengthened this result by showing that a  solution to (\ref{eq1.3}) exists for any diffuse measure $\mu$ and any $f$ satisfying (\ref{eq1.2}) and an integrability condition weaker than (\ref{eq1.5}).

Equations of the form (\ref{eq1.1}) in the case where $L$ is a general, possibly nonlocal, operator associated with a transient regular  Dirichlet form were considered by Klimsiak and Rozkosz  \cite{KR:JFA,KR:CM} in case $f$ satisfies the monotonicity condition, and by Klimsiak \cite{K:PA} in case $f$ satisfies the sign condition (in fact, in \cite{K:PA}  systems of equations  with right-hand side satisfying a generalized sign condition are considered).

In \cite{K:PA,KR:JFA,KR:CM} the proofs of the existence results rely heavily on probabilistic methods. In particular, we make  an extensive use of the theory of backward stochastic differential equations and we use some results from stochastic analysis and probabilistic potential theory. In the present paper we give new, rather short analytical proofs of some of the results of  \cite{K:PA,KR:JFA}. We are motivated  by the desire to make  them accessible
to people working in PDEs that are not familiar with probabilistic methods.

Let $D_e(\EE)$ denote the extended Dirichlet space of $(\EE,D(\EE))$.
In the present paper we provide a proof of the existence of a solution $u$ in the sense of duality (or, equivalently, renormalized solution; see Section \ref{sec3}) to (\ref{eq1.1}) for $f$ satisfying (\ref{eq1.2}) and (\ref{eq1.5}) under the following additional assumption on $\EE$:
\begin{align}
\label{eq1.7} &\mbox{if $\{u_n\}\subset D_e(\EE)$ and $\sup_{n\ge1}\EE(u_n,u_n)<\infty$,}\nonumber \\
&\mbox{\quad then, up to a subsequence, $\{u_n\}$ converges $m$-a.e.}
\end{align}
We also show that if $u$ is a solution to (\ref{eq1.1}), then
$T_k(u)=((-k)\vee u)\wedge k\in  D_e(\EE)$  for every $k>0$ and
\[
\EE(T_k(u),T_k(u))\le 2k\|\mu\|_{TV},
\]
where $\|\mu\|_{TV}$ is the total variation norm of $\mu$. Furthermore, if (\ref{eq1.6}) is satisfied, then the solution $u$ is unique.

Condition (\ref{eq1.7}) holds true  in many interesting situations.
For instance, it holds if
\begin{equation}
\label{eq1.10} \mbox{the embedding $V_1\hookrightarrow L^2(E;m)$ is compact},
\end{equation}
where  $V_1$ denotes the space $D(\EE)$ equipped with the norm determined by the form $\EE_1(\cdot,\cdot):=\EE(\cdot,\cdot)+(\cdot,\cdot)$.
Another condition, which is often satisfied in practice and implies (\ref{eq1.7}), is the so called absolute continuity condition saying that
\begin{equation}
\label{eq1.8}
R_{\alpha}(x,\cdot)\ll m\quad \mbox{for any $\alpha>0$ and $x\in E$},
\end{equation}
where $R_{\alpha}(x,\cdot)$ is the resolvent kernel associated with $\EE$.
For symmetric forms considered in this paper,  condition (\ref{eq1.8}) is equivalent to the condition
\begin{equation}
\label{eq1.9}
P_t(x,\cdot)\ll m\quad \mbox{for any $t>0$ and $x\in E$},
\end{equation}
where $P_t(x,\cdot)$  is the transition kernel associated with $\EE$.

The main idea of our proofs resembles the idea used in  case of problem (\ref{eq1.3}) (see the proof of Theorem B.4 in Brezis, Marcus and Ponce  \cite{BMP} and also Ponce \cite[Chapter 19]{Po}). Let $V$ denote the extended space $D_e(\EE)$ equipped with the norm determined by $\EE$. We first prove the existence of a solution to (\ref{eq1.1}) with  $\mu\in\MM_{0,b}\cap V'$, where $\MM_{0,b}$ is the set of all diffuse measures on $E$ and $V'$ is the dual of $V$. This step can be viewed as  some modification of the result of Brezis and Browder \cite{BB} on absorption problems (\ref{eq1.3}) with $\mu\in H^{-1}(D)$. To get the existence for  general $\mu\in\MM_{0,b}$, we approximate it by a suitably chosen sequence $\{\mu_n\}\subset \MM_{0,b}\cap V'$ and show that solutions $u_n$ of (\ref{eq1.1}) corresponding to the measures $\mu_n$ converge to a solution of (\ref{eq1.1}). In this second step we use some a priori estimates for $u_n$ in $V$ and condition (\ref{eq1.7}).  In \cite{KR:BPAN} it is proved that
any $\mu\in\MM_{0,b}$ admits decomposition of the form $\mu=g+\nu$ with $g\in L^1(E;m)$ and $\nu\in\MM_{0,b}\cap V'$ (this  generalizes  the corresponding result proved  by Boccardo, Gallou\"{e}t and Orsina \cite{BGO}  for the form associated with $\Delta$). Therefore, similarly to \cite{BMP},  in the second step of the proof it is enough  to approximate  by $\{\mu_n\}$ the measure $\mu=g\cdot m$. This, however, does not simplify the reasoning, so in the  present paper we give a  direct approximation of $\mu\in\MM_{0,b}$ (without recourse to \cite{KR:BPAN}).

In the present paper we confine ourselves to single equation with
operator corresponding to symmetric regular Dirichlet forms. For
results (proved with the help of probabilistic methods) for
quasi-regular, possibly nonsymmetric forms, we refer the reader to
\cite{KR:CM}, and for results for systems of equations to
\cite{K:PA}. Also note that equations with $f=0$ but $\Delta$ replaced by the Schr\"odinger operator are treated in \cite{OP2} and \cite[Chapter 22]{Po}.

In the paper we deal exclusively with equations
with diffuse measures. The theory of semilinear equations with general
bounded measures is much more subtle. In this case (\ref{eq1.3}) with
$f$ satisfying (\ref{eq1.6}) need not have a solution
(see \cite{BP,BB1,BMP}).
Results on (\ref{eq1.3}) with general bounded measure $\mu$ and $f$ satisfying the monotonicity condition  are found in \cite{BB1,BMP,DPP}, and for equations
with $f$ satisfying the sign condition (\ref{eq1.2}) in
\cite[Chapter 21]{Po}. The Dirichlet problem for linear equations with nonlocal
operators and bounded measure $\mu$ is studied in
\cite{KPU,KMS,Pe}. In Klimsiak \cite{K:CVPDE} general
equations of the form (\ref{eq1.1}) with general bounded measure
$\mu$ and $f$ satisfying  (\ref{eq1.6}) are considered. The
question whether one can extend the existence results of
{\cite{Po} to some nonlocal operators or extend some existence
results  of \cite{K:CVPDE} to $f$ satisfying (\ref{eq1.2}) remains
open.

\section{Preliminaries}

In this paper,  $E$ is a locally compact separable metric space and $m$ is a Radon measure such that $\mbox{supp}[m]=E$, i.e. $m$ is a nonnegative measure on the $\sigma$-field of Borel subsets of $E$ which is finite on compact sets and strictly positive on nonempty open sets.

In what follows $(\EE,D(\EE))$ is  a symmetric regular Dirichlet form on $L^2(E;m)$.
We denote by $(\cdot,\cdot)$ the usual inner product in $L^2(E;m)$. As usual, for $\lambda\ge0$  we  set $\EE_\lambda(u,v)=\EE(u,v)+\lambda(u,v)$,
$u,v\in D(\EE)$.

In  the whole paper we assume that $(\EE,D(\EE))$ is
transient. Recall that this means that  there exists a bounded strictly $m$-a.e. positive $g\in L^1(E;m)$ such that
\begin{equation}
\label{eq2.5}
\int_E|u(x)|g(x)\,m(dx)\le\EE(u,u),\quad u\in D(\EE).
\end{equation}
For an equivalent formulation, see \cite[Section 1.5]{FOT}. The extended Dirichlet space associated with
$(\EE,D(\EE))$ (see \cite[Section 1.5]{FOT} for the definition)
will be denoted by $(\EE,D_e(\EE))$. Note that $D_e(\EE)$ with the
inner product $\EE$ is a Hilbert space (see \cite[Theorem
1.5.3]{FOT}). In the sequel this space will be denoted by $V$. We denote by
$V'$ the dual space of $V$.
The duality pairing  between $V'$ and $V$ will be denoted by $\langle\langle\cdot,\cdot\rangle\rangle$.

In the paper we define $0$-order quasi notions with respect to $\EE$ (capacity $\mbox{Cap}_{(0)}$, exceptional sets, nests, quasi-continuity) as in \cite[Chapter 2, page 74]{FOT}.  Recall that $\mbox{Cap}_{(0)}$ is defined as follows.  For an open subset $U$ of $E$, we
set
\[
\LL^{(0)}_U=\{u\in D_e(\EE):u\ge1\mbox{ $m$-a.e. on }U\}
\]
and
\[
\mbox{Cap}_{(0)}(U)=\left\{
\begin{array}{ll}
\inf\{\EE(u,u):u\in\LL^{(0)}_U\} &\mbox{if } \LL^{(0)}_U\neq\emptyset,\\
\infty & \mbox{if }\LL^{(0)}_U=\emptyset.
\end{array}
\right.
\]
Then, as usual, for an arbitrary  $A\subset E$, we set
\[
\mbox{Cap}_{(0)}(A)=\inf\{\mbox{Cap}_{(0)}(U):U\mbox{ open, }U\supset
A\}.
\]
We say that $A\subset E$ is exceptional if $\mbox{Cap}_{(0)}(A)=0$, and we
say that a property of points in $E$ holds quasi-everywhere (q.e. in abbreviation) if it holds outside some exceptional subset of $E$.

For a measure $\mu$ on $E$ and a function $u:E\rightarrow\BR$, we
use the notation
\[
\langle\mu,u\rangle=\int_Eu(x)\,\mu(dx),
\]
whenever the integral is well defined.
For a signed Borel measure $\mu$, we denote by $\mu^{+}$ and
$\mu^{-}$ its positive and negative parts, and by $|\mu|$ the total variation measure, i.e. $|\mu|=\mu^{+}+\mu^{-}$. We denote by
$\MM_b$ the space of all finite signed Borel measures on $E$
endowed with the total variation norm $ \|\mu\|_{TV}=|\mu|(E)$,
and by  $\MM_{0,b}$  the subspace of $\MM_b$ consisting
of all measures charging  no set of capacity $\mbox{Cap}_{(0)}$ zero. Elements of $\MM_{0,b}$ are called diffuse measures.

We write  $\mu\in\MM_{0,b}\cap V'$ if for some $c>0$,
\[
\langle|\mu|,|\tilde u|\rangle\le c\EE(u,u)^{1/2},\quad u\in
D_e(\EE),
\]
where $\tilde u$ denotes a quasi-continuous $m$-version of $u$
(see \cite[Theorem 2.1.7]{FOT}). Elements of $\MM_{0,b}\cap V'$
are called measures of finite 0-order energy integral. If
$\mu\in\MM_{0,b}\cap V'$, then
\begin{equation}
\label{eq2.16xy}
\langle\mu,\tilde u\rangle=\langle\langle \mu,u\rangle\rangle,
\quad u\in D_e(\EE).
\end{equation}
If $f\in V'$, then by  Riesz's theorem there is a unique element $Gf\in D_e(\EE)$
such that
\begin{equation}
\label{eq2.16}
\EE(Gf,u)=\langle\langle f,u\rangle\rangle,\quad u\in D_e(\EE).
\end{equation}
In particular, if $\mu\in \MM_{0,b}\cap V'$, then the function $G\mu$ is well defined and belongs to $V$.

Let $\BB^+(E)$ (resp. $\BB_b(E)$) denote the set of all
positive (resp. bounded) real Borel functions on $E$, and let
$(G_\alpha)_{\alpha>0}$ denote the strongly continuous resolvent on $L^2(E;m)$ associated with $(\EE,D(\EE))$. Recall that $\alpha G_{\alpha}$ is Markovian for each $\alpha>0$, i.e. $0\le\alpha G_{\alpha}f\le 1$ $m$-a.e. whenever $f\in L^2(E;m)$ and $0\le f\le1$ $m$-a.e. Since $G_{\alpha}$ is positivity preserving, we can extend it to any positive $f\in\BB^+(E)$ by
\begin{equation}
\label{eq2.17}
G_\alpha f(x)= \lim_{n\rightarrow \infty}G_\alpha f_n(x)=\sup_{n\ge 1}G_\alpha f(x)\quad \mbox{for }m\mbox{-a.e. }x\in E,
\end{equation}
where $\{f_n\}\subset L^2(E;m)$ is a nondecreasing sequence of
positive  functions converging $m$-a.e. to $f$. It is clear that
$G_\alpha f$ does not depend on the choice of the sequence
$\{f_n\}$. By the resolvent equation, if $\beta>\alpha>0$, then
$G_{\alpha}f\le G_{\beta}f$ $m$-a.e. for any $\BB^+(E)$. Therefore
for  $f\in\BB^+(E)$ we can set
\begin{equation}
\label{eq2.18}
Gf(x):=G_0f(x)=\lim_{\alpha\downarrow0}G_\alpha f(x)=\sup_{\alpha>0} G_\alpha f(x)\quad \mbox{for }m\mbox{-a.e. }x\in E.
\end{equation}
By \cite[Lemma 2.2.11]{FOT}, for $f\in\BB^+(E)$ such that $f\cdot
m\in V'$, $Gf$ defined by (\ref{eq2.18}) coincides with $Gf$ of
(\ref{eq2.16}).

An increasing sequence $\{F_n\}$ of closed subsets of $E$ is called a generalized nest if $\mbox{Cap}_{(0)}(K\setminus F_n)\rightarrow 0$ for any compact $K\subset E$. A Borel measure $\mu$ on $E$
is called smooth if there exists a generalized nest $\{F_n\}$ such that  $\mathbf{1}_{F_n}\cdot\mu\in \MM_{0,b},\, n\ge 1$. In particular each diffuse  measure is smooth. If $\{F_n\}$ is a generalized  nest such that  $\mathbf{1}_{F_n}\cdot\mu\in \MM_{0,b}$ then
\begin{equation}
\label{eq2.8}
\mu\Big(E\setminus\bigcup^{\infty}_{n=1}F_n\Big)=0
\end{equation}
(see \cite[(2.2.18)]{FOT}).

By  the 0-order version of \cite[Theorem 2.2.4]{FOT} (see the remark following \cite[Corollary
2.2.2]{FOT}),  for each smooth measure $\mu$
there exists a generalized nest $\{F_n\}$ such that $\mathbf{1}_{F_n}\cdot\mu\in\MM_{0,b}\cap V'$.
Therefore, for a positive  smooth measure $\mu$, we may define  a function $G\mu$ with values in $[0,\infty]$ by
\begin{equation}
\label{eq2.14}
G\mu(x)=\lim_{n\rightarrow\infty} G\mu_n(x)=\sup_{n\ge 1} G\mu_n(x),\quad x\in E,
\end{equation}
where $\mu_n=\mathbf{1}_{F_n}\cdot\mu$ and $\{F_n\}$ is a generalized nest such that $\mu_n\in\MM_{0,b}\cap V'$, $n\ge1$. Note that $G\mu$ is defined uniquely up to $m$-equivalence.

In the paper, for a function $u$ on $E$, we denote by $\tilde
u$ its quasi-continuous $m$-version (whenever it exists). We will
freely  use, without explicit mention,  the following fact: if
$u_1\le u_2$ $m$-a.e. and $u_1, u_2$ have quasi-continuous
$m$-versions, then $\tilde u_1\le \tilde u_2$ q.e. (see
\cite[Lemma 2.1.4]{FOT}).

Let $\beta>0$. In the proof of the lemma below we will need the symmetric form $\EE^{(\beta)}$ defined by
\[
\EE^{(\beta)}(u,v)=\beta(u,v-\beta G_{\beta}v),\quad u,v\in L^2(E;m).
\]
Since $\beta G_{\beta}$ is a symmetric linear operator on $L^2(E;m)$, by \cite[Lemma 1.4.1]{FOT} there exists a unique nonnegative symmetric Radon measure $\sigma$
on the product space $E\times E$ such that for any Borel functions $u,v\in L^1(E;m)$,
\begin{equation}
\label{eq2.9}
(u,\beta G_{\beta})=\int_{E\times E}u(x)v(x)\,\sigma_{\beta}(dx\,dy).
\end{equation}
Since $\beta G_{\beta}$ is Markovian, from (\ref{eq2.9}) it follows that $\sigma_{\beta}(E\times B)\le m(B)$ for any Borel  $B\subset E$. Let $s_{\beta}$ denote the Radon-Nikodym derivative of the measure $B\mapsto\sigma_{\beta}(E\times B)$ with respect to $m$. Then $0\le s_{\beta}\le1$ $m$-a.e., and by a direct computation one can check that for a Borel $u\in L^2(E;m)$ one can rewrite $\EE^{(\beta)}(u,u)$ in the form
\begin{align}
\label{eq2.21} \EE^{(\beta)}(u,u)&=\frac{\beta}{2}\int_{E\times E}
(u(x)-u(y))^2\sigma_{\beta}(dx\,dy)\nonumber\\
&\quad+\beta\int_E(u(x))^2(1-s_{\beta}(x))\,m(dx)
\end{align}
(see \cite[(1.4.8)]{FOT}). The expression (\ref{eq2.21}) can be extended to
any Borel function
$u$ on $E$. Furthermore,  by \cite[Theorem 1.5.2(ii)]{FOT}, for
any Borel $u\in D_e(\EE)$, $\EE^{(\beta)}(u,u)$ increases to
$\EE(u,u)$ as $\beta\rightarrow\infty$.

\begin{lemma}
\label{lem2.2}
Let $u\in D_e(\EE)$, and let $\psi:\BR\rightarrow\BR$ be an increasing function such that $\psi(0)=0$ and $\psi$ is Lipschitz continuous with Lipschitz constant 1. Then
\[
\EE(\psi(u),\psi(u))\le \EE(u,\psi(u)).
\]
\end{lemma}
\begin{proof}
By \cite[Theorem 1.5.3$(\delta)$]{FOT}, $\psi(u)\in D_e(\EE)$. Furthermore, for any $\beta>0$ we have
\begin{align*}
\EE^{(\beta)}(u,\psi(u))&=\frac14\{\EE^{(\beta)}(u+\psi(u),u+\psi(u))
-\EE^{(\beta)}(u-\psi(u),u-\psi(u))\}\\
&=\frac{\beta}{2}\int_{E\times E}(u(x)-u(y))(\psi(u(x))-\psi(u(y)))
\sigma_{\beta}(dx\,dy)\\
&\quad+\beta\int_Eu(x)\psi(u(x))(1-s_{\beta}(x))\,m(dx)
\end{align*}
and
\begin{align*}
\EE^{(\beta)}(\psi(u),\psi(u))&=\frac{\beta}{2}
\int_{E\times E} (\psi(u(x))-\psi(u(y)))^2\sigma_{\beta}(dx\,dy) \\
&\quad+\beta\int_E(\psi(u(x)))^2(1-s_{\beta}(x))\,m(dx).
\end{align*}
From this we conclude that for any $\beta>0$,
\[
\EE^{(\beta)}(\psi(u),\psi(u))\le \EE^{(\beta)}(u,\psi(u)).
\]
Letting $\beta\rightarrow\infty$ and using \cite[Theorem
1.5.2(ii)]{FOT} we obtain the desired inequality.
\end{proof}

For $k\ge0$ and $u:E\rightarrow\BR$, we write
\[
T_k(u)(x)=((-k)\vee u(x))\wedge k,\quad x\in E.
\]
Since $\psi(y)=((-k)\vee y)\wedge y$, $y\in\BR$, satisfies the assumptions of Lemma \ref{lem2.2},
if $u\in D_e(\EE)$, then $T_ku\in D_e(\EE)$ and for every $k\ge0$,
\begin{equation}
\label{eq2.22}
\EE(T_k(u),T_k(u))\le \EE(u,T_k(u)).
\end{equation}

\begin{lemma}
\label{lem2.3}
Let $\mu\in\MM_{0,b}^+$. Then $G\mu$ has a quasi-continuous $m$-version.
\end{lemma}
\begin{proof}
Let $\{H_n\}$ be a generalized nest  such that $\mu_n=\mathbf{1}_{H_n}\mu\in\MM_{0,b}\cap V'$, $n\ge 1$,
and $\mu(E\setminus H_n)\le 2^{-3n}$.
Set $u_n=G\mu_n,\, u=G\mu$. It is clear that $u_n\nearrow u$ $m$-a.e. By the $0$-order version of \cite[(2.1.10)]{FOT}, for all  $\varepsilon,\delta>0$ we have
\begin{align*}
\mbox{Cap}_{(0)}(\tilde u_{n+1}-\tilde u_n>\varepsilon)
&=\mbox{Cap}_{(0)}(T_{\varepsilon+\delta}(\tilde u_{n+1}-\tilde u_n)>\varepsilon)\\
&\le \varepsilon^{-2}\EE(T_{\varepsilon+\delta}(u_{n+1}-u_n),T_{\varepsilon+\delta}(u_{n+1}-u_n))
\end{align*}
Since $\tilde u_{n+1}-\tilde u_n\in D_e(\EE)$, it follows from the above inequality and (\ref{eq2.22}) that
\begin{align*}
\nonumber \mbox{Cap}_{(0)}(\tilde u_{n+1}-\tilde u_n>\varepsilon)&
\le \varepsilon^{-2}\EE(u_{n+1}-u_n,T_{\varepsilon+\delta}(u_{n+1}-u_n))\nonumber\\
&=\varepsilon^{-2}\int_E T_{\varepsilon+\delta}(\tilde u_{n+1}-\tilde u_n)\,d(\mu_{n+1}-\mu_n)
\nonumber\\
&\le(\varepsilon+\delta)\varepsilon^{-2}\mu(E\setminus H_n).
\end{align*}
Taking $\varepsilon=2^{-n}$ and letting $\delta\searrow 0$ we get
\begin{equation}
\label{eq2.11}
\mbox{Cap}_{(0)}(\tilde u_{n+1}-\tilde u_n>2^{-n})\le 2^{-2n},\quad n\ge 1.
\end{equation}
By \cite[Theorem 2.1.2]{FOT}, there exists a  nest $\{G_k\}$ such that   $\tilde u_n$ is continuous on $G_k$
for all $k,n\ge 1$. Let $F_n=\bigcap_{k=n}^{\infty}(E\setminus U_k)\cap G_k$, where $U_k=\{\tilde u_{k+1}-\tilde u_k>2^{-k}\}$. From (\ref{eq2.11}) it is clear that $\{F_n\}$
is a nest and $\tilde u$ defined q.e. as $\tilde u= \lim_{n\rightarrow \infty}\tilde{u}_n$ is quasi-continuous. Of course, $\tilde u$ is an $m$-version of $u$.
\end{proof}

\begin{lemma}
\label{lem2.0} Let $\mu\in\MM_{0,b}$. Then for every $k\ge0$,
$T_k(G\mu)\in D_e(\EE)$ and
\begin{equation}
\label{eq2.20}
\EE(T_k(G\mu),T_k(G\mu))\le k\|\mu\|_{TV}.
\end{equation}
\end{lemma}
\begin{proof}
By Lemma \ref{lem2.3}, $G\mu^+, G\mu^-$ are finite $m$-a.e., so $G\mu$ is well defined $m$-a.e.  Let $\{F_n\}$ be a generalized nest for $\mu$ such that $\mu_n=\mathbf{1}_{F_n}\mu\in V'$, $n\ge 1$.
Set $u_n=G\mu_n$, $u=G\mu$. Then $u_n\in D_e(\EE)$ and  by (\ref{eq2.16xy}) and (\ref{eq2.16}),
\[
\EE(u_n,T_k(u_n))=\int_E \widetilde{T_k(u_n)}(x)\,\mu_n(dx)\le k\|\mu\|_{TV}.
\]
By  (\ref{eq2.22}),
$\EE(T_k(u_n),T_k(u_n))\le \EE(u_n,T_k(u_n))$. Hence
\begin{equation}
\label{eq2.19}
\EE(T_k(u_n),T_k(u_n))\le k\|\mu\|_{TV}.
\end{equation}
In particular, $\{T_k(u_n)\}_n$ is weakly relatively compact in $V$. Taking a subsequence if necessary, we can assume that
$T_k(u_n)\rightarrow v$ weakly in $V$ as $n\rightarrow\infty$. By the Banach-Saks theorem, there is a subsequence $(n_l)$ such that the Ces\`aro mean $\{v_N:=\frac1{N}\sum_{l=1}^NT_k(u_{n_l})\}$ converges strongly to $v$ in $V$. Hence,  by (\ref{eq2.5}), $v_N\rightarrow v$ in $L^1(E;g\cdot m)$. On the other hand, $u_n\rightarrow u$ $m$-a.e., so $v_N\rightarrow T_k(u)$ $m$-a.e. Consequently, $T_k(u_n)\rightarrow T_k(u)$ weakly in $V$ as $n\rightarrow\infty$. Therefore
letting $n\rightarrow\infty$ in (\ref{eq2.19}) yields (\ref{eq2.20}).
\end{proof}

\begin{lemma}
\label{lem.eng1}
Let $\mu\in\MM_{0,b}^+$, and let $\widetilde{G\mu}$ be a quasi-continuous $m$-version of $G\mu$.
Then for every $\varepsilon>0$,
\[
\mbox{\rm Cap}_{(0)}(\widetilde{G\mu}>\varepsilon)\le \varepsilon^{-1}\|\mu\|_{TV}.
\]
\end{lemma}
\begin{proof}
By Lemma \ref{lem2.0}, $T_k(\widetilde{G\mu})\in D_e(\EE)$ and (\ref{eq2.20}) holds true.
By this and  the $0$-order version of \cite[(2.1.10)]{FOT}, for all  $\varepsilon,\delta>0$ we have
\begin{align*}
\mbox{Cap}_{(0)}(\widetilde{G\mu}>\varepsilon)
=\mbox{Cap}_{(0)}(T_{\varepsilon+\delta}(\widetilde{G\mu})>\varepsilon)
&\le \varepsilon^{-2}\EE(T_{\varepsilon+\delta}(\widetilde{G\mu}),
T_{\varepsilon+\delta}(\widetilde{G\mu}))\\
&\le\varepsilon^{-2}(\varepsilon+\delta)\|\mu\|_{TV},
\end{align*}
which implies the desired inequality.
\end{proof}

\begin{lemma}
\label{lem2.1}
Assume that $\{\mu_n\}\subset\MM_{0,b}^+$ and $\|\mu_n\|_{TV}\rightarrow 0$.
Then, up to a subsequence, $\widetilde{G\mu_n}\rightarrow 0$ q.e.
\end{lemma}
\begin{proof}
We can and do  assume that $\|\mu_n\|\le 2^{-2n}$, $n\ge 1$. Then by Lemma \ref{lem.eng1},
\[
\mbox{\rm Cap}_{(0)}(\widetilde{G\mu_n}>2^{-n})\le 2^{-n}.
\]
Let $F=\bigcup_{n\ge 1}\bigcap_{k\ge n} \{\widetilde{G\mu_k}\le 2^{-k}\}$. By the above inequality, $\mbox{\rm Cap}_{(0)}(E\setminus F)=0$. This proves the lemma because by the definition of $F$, $\widetilde{G\mu_n}\rightarrow 0$ q.e. on $F$.
\end{proof}

\begin{lemma}
There exists a strictly positive function $g\in\BB(E)$ such that $\|Gg\|_\infty<\infty$.
\end{lemma}
\begin{proof}
Since $\EE$ is transient there exists a strictly positive  $h\in L^1(E;m)$ such that $Gh<\infty$. By \cite[Theorem 2.2.4]{FOT}, there exist a nest $\{F_n\}$ such that $\mathbf{1}_{F_n}\cdot h\in L^1(E;m)\cap V'$. Write $H_{n,k}=\{G(\mathbf{1}_{F_n}h)\le k\}$ and $H_k=\{Gh\le k\}$. By \cite[Lemma 2.2.4]{FOT}, $G(\mathbf{1}_{F_n\cap H_{n,k}}h)\le k$. Letting $n\rightarrow \infty$ yields $G(\mathbf{1}_{H_{k}}h)\le k$. Set $g=\sum_{n=0}^{\infty}\frac{1}{2^n(n+1)}\mathbf{1}_{H_n}h$. Then
\[
Gg=\sum_{n=0}^\infty \frac{1}{2^n(n+1)}\,G(\mathbf{1}_{H_n}h)\le \sum_{n=0}^{\infty}2^{-n},
\]
which proves the lemma.
\end{proof}

\begin{lemma}
\label{lm2.n2} For any positive $\eta\in L^1(E;m)$ such that
$\|G\eta\|_{\infty}<\infty$ and any positive $\mu\in \MM_{0,b}$,
\begin{equation}
\label{eq2.lmab}
\langle\mu, \widetilde {G\eta}\rangle=\int_E\eta(x)G\mu(x)\,m(dx),
\end{equation}
where $\widetilde {G\eta}$ is a quasi-continuous $m$-version of $G\eta$.
\end{lemma}
\begin{proof}
We first assume that $\mu\in\MM_{0,b}\cap V'$ and $\eta\in L^1(E;m)\cap V'$. Then
\[
\EE(G\mu,G\eta)=\int_E \widetilde{G\eta}(x)\,\mu(dx),\qquad
\EE(G\eta,G\mu)=\int_E\eta(x) G\mu(x)\,m(dx).
\]
Since $\EE$ is symmetric, this implies  (\ref{eq2.lmab}). Now assume that $\mu\in\MM_{0,b}$, $\eta\in L^1(E;m)$ and   $G\eta$
is bounded. Let $\{F_n\}$ be a generalized nest such that $\eta_n=\mathbf{1}_{F_n}\cdot\eta\in L^1(E;m)\cap V'$ and $\mu_n=\mathbf{1}_{F_n}\cdot\mu\in \MM_{0,b}\cap V'$. By what has already been proved,
\[
\langle\mu_n, \widetilde {G\eta_n}\rangle=\int_E\eta(x) G\mu_n(x)\,m(dx), \quad n\ge1.
\]
Letting $n\rightarrow\infty$ we get (\ref{eq2.lmab}).
\end{proof}

In the rest of this section we assume that the absolutely continuity condition (\ref{eq1.8}) is satisfied. Condition (\ref{eq1.8}) was introduced by P.-A. Meyer \cite{M}. It is sometimes called  condition (L) (see \cite[p. 246]{DM}). By \cite[Theorem 4.2.4]{FOT}, condition  (\ref{eq1.8}) is equivalent to (\ref{eq1.9}). If (\ref{eq1.8})  is satisfied, then for any $\alpha>0$ there  exists a positive $\BB(E)\otimes\BB(E)$-measurable function $r_{\alpha}:E\times
E\rightarrow\BR$ such that $r_{\alpha}(x,y)=r_{\alpha}(y,x)$, $x,y\in E$, and for any $f\in\BB^+(E)$,
\begin{equation}
\label{eq2.6}
G_{\alpha}f(x)=\int_Er_{\alpha}(x,y)f(y)\,m(dy)\quad\mbox{for } m \mbox{-a.e. }  x\in E.
\end{equation}
Moreover, there exists a positive $\BB(E)\otimes\BB(E)$-measurable function $r:E\times E\rightarrow\BR$ such that $r(x,y)=r(y,x)$, $x,y\in E$, and for any $f\in\BB^+(E)$,
\[
Gf(x)=\int_Er(x,y)f(y)\,m(dy),\quad\mbox{for } m \mbox{-a.e. }  x\in E.
\]
In fact, $r(x,y)=\lim_{\alpha\downarrow0}r_{\alpha}(x,y)$ (see the remarks in \cite[p. 256]{BG}).

\begin{lemma}
\label{lm2.n1} Assume that \mbox{\rm(\ref{eq1.8})} is satisfied.
If $\mu\in \MM_{0,b}$, then for $m$-a.e. $x\in E$,
\begin{equation}
\label{eq2.13} G\mu(x)=\int_Er(x,y)\,\mu(dy).
\end{equation}
\end{lemma}
\begin{proof}
Let $\{F_n\}$ be a generalized nest such that $\mu_n=\mathbf{1}_{F_n}\cdot\mu\in \MM_{0,b}\cap V'$. By
\cite[Exercise 4.2.2, Lemma 5.1.3]{FOT}, for any $\alpha>0$ we have
\[
G_{\alpha}\mu_n(x)=\int_Er_{\alpha}(x,y)\,\mu_n(dy)
\]
for $m$-a.e. $x\in E$. Letting $\alpha\downarrow0$ in the above equality yields (\ref{eq2.13}) with $\mu$ replaced by $\mu_n$.  Then, using (\ref{eq2.8}), (\ref{eq2.14}) and the monotone convergence, we get (\ref{eq2.13}) for $\mu$.
\end{proof}

\section{Existence and uniqueness of solutions}
\label{sec3}

Throughout this section, we  assume that $\mu\in\MM_{0,b}$ and $f:E\times\BR\rightarrow\BR$ is
a Carath\'eodory function, i.e.  $f(\cdot,y)$ is measurable on $E$ for each fixed $y\in\BR$, and $f(x,\cdot)$ is continuous on $\BR$ for each fixed $x\in E$.

Following \cite{KR:JFA} we adopt the following definition.
\begin{definition}
\label{def3.1} We say that $u:E\rightarrow\BR$ is a solution of
(\ref{eq1.1}) (in the sense of duality) if
\begin{enumerate}
\item[(a)]$f(\cdot,u)\in L^1(E;m)$,
\item[(b)]for any $\eta\in L^1(E;m)$ such that $G|\eta|$ is bounded  we have
\begin{equation}
\label{eq2.3}
\int_E u(x)\eta(x)\,m(dx)=\int_E f(x,u(x))G\eta(x)\,m(dx)
+\int_E \widetilde{G\eta}(x)\,\mu(dx).
\end{equation}
\end{enumerate}
\end{definition}


\begin{remark}
If $u$ is a solution of (\ref{eq1.1}), then $u$ has a
quasi-continuous $m$-version, because then $u=G(f(\cdot,u)\cdot
m+\mu)$ $m$-a.e. by Lemma \ref{lm2.n2}, so the existence of a
quasi-continuous $m$-version follows from  Lemma \ref{lem2.3}.
\end{remark}

Recall that an increasing sequence $\{F_n\}$ of closed subsets of $E$ is called a generalized nest if $\mbox{Cap}_{(0)}(K\setminus F_n)\rightarrow 0$ for any compact $K\subset E$.
\begin{proposition}
\label{prop.equiv}
Let $u$ be a measurable function such that $f(\cdot,u)\in L^1(E;m)$. Then the following assertions
are equivalent:
\begin{enumerate}
\item[\rm(i)] $u$ is a solution to \mbox{\rm(\ref{eq1.1})}.

\item[\rm(ii)] $u=G(f(\cdot,u))+G\mu$ $m$-a.e.

\item[\rm(iii)] For any generalized nest $\{F_n\}$ such
that $\mathbf{1}_{F_n}\cdot(f(\cdot,u)\cdot m+\mu)\in\MM_{0,b}\cap V'$ we have $u_n\rightarrow u$ $m$-a.e., where $u_n\in D_e(\EE)$ is the unique solution of the problem
\begin{equation}
\label{star1}
\EE(u_n,\eta)=\langle\mathbf{1}_{F_n}f(\cdot,u)\cdot m,\eta\rangle+\langle\mathbf{1}_{F_n}\cdot\mu,\tilde \eta\rangle,
\quad \eta\in D_e(\EE)
\end{equation}
\item[\rm(iv)] For some generalized nest $\{F_n\}$ such that
$\mathbf{1}_{F_n}\cdot(f(\cdot,u)\cdot m+\mu)\in\MM_{0,b}\cap V'$
we have $u_n\rightarrow u$ $m$-a.e., where $u_n\in D_e(\EE)$ is
the unique solution of \mbox{\rm(\ref{star1})}.

\item[\rm(v)]  For any generalized nest $\{F_n\}$ such
that $\mathbf{1}_{F_n}\cdot(f(\cdot,u)\cdot m+\mu)
\in\MM_{0,b}\cap V'$ we have $\tilde u_n\rightarrow\tilde  u$
q.e., where $u_n\in D_e(\EE)$ is the unique solution of
\mbox{\rm(\ref{star1})}.

\item[\rm(vi)]  For some generalized nest $\{F_n\}$ such
that $\mathbf{1}_{F_n}\cdot(f(\cdot,u)\cdot m+\mu)
\in\MM_{0,b}\cap V'$ we have $\tilde u_n\rightarrow\tilde  u$
q.e., where $u_n\in D_e(\EE)$ is the unique solution of
\mbox{\rm(\ref{star1})}.
\end{enumerate}
\end{proposition}
\begin{proof}
The equivalence of (i) and (ii) follows from Lemma \ref{lm2.n2}. Obviously, (iii) implies (iv), (v) implies (iii) and (vi),  and (vi) implies  (iv). What is left ist to show that (ii) implies (v) and (iv) implies (ii). Let $\{F_n\}$ be a generalized nest for $f(\cdot,u)\cdot m+\mu$, and let
\begin{equation}
\label{eq2.sol}
u_n=G(\mathbf{1}_{F_n}f(\cdot,u))+G(\mathbf{1}_{F_n}\cdot\mu).
\end{equation}
By the definition of $\{F_n\}$, $u_n\in V$. Moreover, by (\ref{eq2.16xy}) and (\ref{eq2.16}),  $u_n$ satisfies (\ref{star1}).
If (ii) is satisfied, then $|\tilde u-\tilde u_n|\le \widetilde{G\mathbf{1}_{E\setminus F_n}|f(\cdot,u)|}+\widetilde{G\mathbf{1}_{E\setminus F_n}\cdot|\mu|}$ q.e. Hence, by Lemma \ref{lem2.1}, $\tilde u_n\rightarrow \tilde u$ q.e. Now assume  (iv). By (\ref{eq2.16}), since $u_n$ solves (\ref{star1}), it  is given by  (\ref{eq2.sol}). Therefore letting $n\rightarrow \infty$
in (\ref{eq2.sol}) we get (ii).
\end{proof}

\begin{remark}
\label{rem3.2}
Let $f(\cdot,u)\in L^1(E;m)$. By Lemma \ref{lm2.n1} and Proposition \ref{prop.equiv}, if (\ref{eq1.8}) is satisfied, then  $u$ is a solution to (\ref{eq1.1})  if and only if
\[
u(x)=\int_Ef(y,u(y)) r(x,y)\,m(dy)+\int_E r(x,y)\,\mu(dy)
\]
for $m$-a.e. $x\in E$.
\end{remark}

\begin{remark}
In \cite{KR:NoD} the following definition of a solution of (\ref{eq1.1}) is introduced: $u:E\rightarrow\BR$ is a renormalized solution of (\ref{eq1.1}) if
\begin{enumerate}
\item[(a)] $f(\cdot,u)\in L^1(E;m)$
and $T_k(u)\in D_e(\EE)$ for every $k>0$,
\item[(b)] there exists a sequence
$\{\nu_k\}\subset\MM_{0,b}(E)$ such that
$\|\nu_{k}\|_{TV}\rightarrow0$ as $k\rightarrow\infty$ and for
every $k\in\BN$ and every bounded $v\in D_e(\EE)$,
\[
\EE(T_k(u),v) =\langle f(\cdot,u)\cdot m+\mu,\tilde v\rangle+
\langle\nu_{k},\tilde v\rangle.
\]
\end{enumerate}
Note that in case of local operators, this is essentially
\cite[Definition 2.29]{DMOP}. By  \cite[Proposition 5.3]{KR:JFA}
and \cite[Theorem 3.5]{KR:NoD}, $u$ is a solution of (\ref{eq1.1})
in the sense of Definition \ref{def3.1}  if and only if it is a
renormalized solution.
\end{remark}

\begin{lemma}
\label{lem3.1}
\begin{enumerate}
\item[\rm(i)]
Let $u$ be a solution of \mbox{\rm(\ref{eq1.1})} with $f$ satisfying \mbox{\rm(\ref{eq1.2})}. Then for every $a>0$,
\[
\int_{\{|u|> a\}} |f(x,u(x)) |\,m(dx)\le\|\fch_{\{|\tilde u|>a\}}
\cdot\mu\|_{TV}.
\]
\item[\rm(ii)]Assume that $f$ satisfies \mbox{\rm(\ref{eq1.6})}.  If $u_i$,
$i=1,2$, is a solution of \mbox{\rm(\ref{eq1.1})} with  $\mu$
replaced by $\mu_i\in\MM_{0,b}$, then
\[
\|f(\cdot,u_1)-f(\cdot,u_2)\|_{L^1(E;m)} \le\|\mu_1-\mu_2\|_{TV}.
\]
\end{enumerate}
\end{lemma}
\begin{proof}
Let $\{F_n\}$ be a generalized nest such that
$\fch_{F_n}(|f(\cdot,u)|\cdot m+|\mu|)\in\MM_{0,b}\cap V'$. For
$n\ge1$ we set $f_n=\fch_{F_n}f(\cdot,u)$,
$\mu_n=\fch_{F_n}\cdot\mu$ and  $u_n=G(f_n\cdot m+\mu_n)$. Then
$u_n\in D_e(E)$.
For $a>0$, $k\in\BN$ we set
\[
\psi_{a,k}(y)=\frac{k(y-a)^+}{1+k(y-a)^+}-\frac{k(y+a)^-}{1+k(y+a)^-}\,,\quad y\in\BR.
\]
Since $\psi:=(1/k)\psi_{a,k}$ satisfies the assumptions of Lemma \ref{lem2.2},
$\psi_{a,k}(u_n)\in D_e(\EE)$ and
\[
\EE(u_n,\psi_{a,k}(u_n))\ge0
\]
Let $\tilde u_n$ be a quasi-continuous $m$-version of $u_n$. Then $\psi_{a,k}(\tilde u_n)$ is a quasi-continuous $m$-version of $\psi_{a,k}(u_n)$.
By Proposition \ref{prop.equiv},
\[
\EE(u_n,\psi_{a,k}(u_n))=\langle f_n\cdot m,\psi_{a,k}(u_n)\rangle +\langle\mu_n,\psi_{a,k}(\tilde u_n)\rangle.
\]
Hence
\begin{equation}
\label{eq3.37}
-\int_E f_n(x)\psi_{a,k}(u_n(x))\,m(dx)
\le \int_E\psi_{a,k}(\tilde u_n(x))\,\mu_n(dx) \\
\le \int_{\{|\tilde u_n|>a\}}\,|\mu|(dx).
\end{equation}
By Proposition \ref{prop.equiv}(v), $\tilde u_n\rightarrow\tilde u$ q.e. Therefore letting
$n\rightarrow\infty$ in (\ref{eq3.37}) and using the dominated convergence theorem we obtain
\[
-\int_E f(x, u(x))\psi_{a,k}(u(x))\,m(dx)\le\int_{\{|\tilde u|>a\}}\,|\mu|(dx).
\]
By (\ref{eq1.2}) and the definition of $\psi_{a,k}$, $|f(\cdot,u)\psi_{a,k}(u)|=-f(\cdot,u)\psi_{a,k}(u)$. Hence
\[
\int_E |f(x,u(x))||\psi_{a,k}(u(x))|\,m(dx)\le  \int_{\{|\tilde u|>a\}}\,|\mu|(dx).
\]
Letting $k\rightarrow\infty$ in the above inequality  yields part (i) of the lemma. To get (ii), we observe that $v= u_1-u_2$ is a solution to the problem
\[
-Lv=g(\cdot,v)+\mu_1-\mu_2
\]
with $g(x,y)=f(x,y+u_2(x))-f(x,u_2(x))$. Since $f$ satisfies
(\ref{eq1.6}), $g$ satisfies (\ref{eq1.2}). Therefore  the desired
inequality follows from part (i).
\end{proof}

Note that from Lemma \ref{lem3.1}(i) with $a=0$ the following absorption estimate follows:
\begin{equation}
\label{eq3.33}
\|f(\cdot,u)\|_{L^1(E;m)}\le \|\mu\|_{TV}.
\end{equation}

\begin{corollary}
If $f$ satisfies \mbox{\rm(\ref{eq1.6})}, then there exists at
most one solution to \mbox{\rm(\ref{eq1.1})}.
\end{corollary}
\begin{proof}
Let $u_1$, $u_2$ be  solutions of \mbox{\rm(\ref{eq1.1})}, and let
$v=u_1-u_2$. By Lemma \ref{lem3.1}(ii), $v$ is a solution of the
problem $-Lv=0$. Hence $v=0$ $m$-a.e. by Proposition
\ref{prop.equiv}(ii).
\end{proof}



\begin{proposition}
\label{prop3.3} If $u$ is a solution of \mbox{\rm(\ref{eq1.1})} with $f$ satisfying \mbox{\rm(\ref{eq1.2})}, then for every $k\ge0$, $T_k(u)\in D_e(\EE)$ and
\begin{equation}
\label{eq3.35}
\EE(T_k(u),T_k(u))\le 2k\|\mu\|_{TV}.
\end{equation}
\end{proposition}
\begin{proof}
By Proposition \ref{prop.equiv}(ii) and Lemma \ref{lem2.0}, $T_k(u)\in D_e(\EE)$ and
\[
\EE(T_k(u),T_k(u))\le k(\|f(\cdot,u)\|_{L^1(E;m)}+\|\mu\|_{TV}),
\]
which when combined with (\ref{eq3.33}) yields (\ref{eq3.35}).
\end{proof}

\begin{lemma}
\label{lem2.5} Let $\mu\in\MM_b$. If $|\mu|$ charges no set of capacity $\mbox{\rm Cap}_{(0)}$ zero, then for every $\varepsilon>0$  there exists $\delta>0$ such that for any Borel subset $B$ of $E$, if $\mbox{\rm Cap}_{(0)}(B)\le\delta$, then $|\mu|(B)\le\varepsilon$.
\end{lemma}
\begin{proof}
By the 0-order version of \cite[Lemma 2.1.2]{FOT} (see the remarks following \cite[(2.1.14)]{FOT}) and  \cite[Theorem A.1.2]{FOT}, $\mbox{Cap}_{(0)}$ is a countably subadditive set function. Therefore the desired result follows from \cite[Proposition 14.7]{Po}.
\end{proof}

\begin{theorem}
\label{th3.5} Assume  \mbox{\rm(\ref{eq1.7})}.  If $f$ satisfies \mbox{\rm(\ref{eq1.2})} and \mbox{\rm(\ref{eq1.5})}, then there exists a solution of \mbox{\rm(\ref{eq1.1})}. Moreover, for every $k\ge0$, $T_k(u)\in D_e(\EE)$ and \mbox{\rm(\ref{eq3.35})} is satisfied.
\end{theorem}
\begin{proof}
We divide the proof into two steps.\\
{\em Step 1.}  We first assume that $\mu^+\mu^{-}\in\MM_{0,b}\cap
V'$. For a positive  $g\in L^2(E;m)\cap V'$ set $f_n=\frac{ng}{1+ng}T_n(f)$, $n\in\BN$.
Under the hypothesis (\ref{eq1.5}) the operator $A_n:V\rightarrow V'$ defined as $A_n(u)=-Lu-f_n(u)$
is pseudomonotone (see, e.g., \cite[Section 2.1]{R} for the definition).
Indeed, it is clear that $A_n$ maps bounded sets of $V$ into bounded sets of $V'$. 
Next, suppose that $u_k\rightarrow u$ weakly in $V$. Then for any $v\in V$,
\begin{align*}
\liminf_{k\rightarrow\infty}\langle\langle-Lu_k,u_k-v\rangle\rangle
&=\liminf_{k\rightarrow\infty}\EE(u_k,u_k-v)
=\liminf_{k\rightarrow\infty}\EE(u_k,u_k)-\EE(u_k,v)\\
&\ge\EE(u,u)-\EE(u,v)=\langle\langle-Lu,u-v\rangle\rangle.
\end{align*}
Furthermore, by (\ref{eq1.7}), we can assume that $u_k\rightarrow u$ $m$-a.e.
Consequently, we can assume that $f_n(\cdot,u_k)\rightarrow f_n(\cdot,u)$ in $V'$.
Therefore, for  any $v\in V$,
\begin{align*}
\lim_{k\rightarrow\infty}\langle\langle-f_n(\cdot,u_k),u_k-v\rangle\rangle
&=\lim_{k\rightarrow\infty}(-f_n(\cdot,u_k),u_k-v)\\
&=\liminf_{k\rightarrow \infty}\int_E|f_n(x,u_k(x)) u_k(x)|\,m(dx)
+\lim_{k\rightarrow \infty}\langle\langle f_n(\cdot,u_k),v\rangle\rangle\\
&\ge\int_E|f_n(x,u(x)) u(x)|\,m(dx)+\langle\langle f_n(\cdot,u),v\rangle\rangle\\&=(-f_n(\cdot,u),u-v)=\langle\langle-f_n(\cdot,u),u-v\rangle\rangle.
\end{align*}
Accordingly, $A_n$ is pseudomonotone. Since by (\ref{eq1.2}), for  $u\in V$ we have $\langle\langle A_nu,u\rangle\rangle=\EE(u,u)-(f_n(\cdot,u),u)\ge\EE(u,u)$, the operator $A_n$ is also coercive. Therefore $A_n$ is surjective by standard result in the theory of pseudomonotone mappings (see, e.g., \cite[Theorem 2.6]{R}).
Thus, there exists a weak solution $u_n\in D_e(\EE)$ of the equation
\begin{equation}
\label{eq3.7}
-Lu_n=f_n(\cdot,u_n)+\mu,
\end{equation}
i.e. for any $v\in D_e(\EE)$,
\begin{equation}
\label{eq3.1}
\EE(u_n,v)=\int_Ef_n(x,u_n(x))v(x)\,m(dx)+\langle\langle\mu,v\rangle\rangle.
\end{equation}
Taking $u_n$ as a test function in (\ref{eq3.1}) we get
\begin{equation}
\label{eq3.2} \EE(u_n,u_n)-\int_Ef_n(x,u_n(x))u_n(x)\,m(dx)
=\langle\langle\mu,u_n\rangle\rangle\\
\le\|\mu\|_{V'}\EE(u_n,u_n)^{1/2}.
\end{equation}
By (\ref{eq1.2}) and (\ref{eq3.2}),
\begin{equation}
\label{eq3.3} \EE(u_n,u_n)+\int_E|f_n(x,u_n(x))u_n(x)|\,m(dx)
\le\|\mu\|^2_{V'},\quad n\ge 1.
\end{equation}
By (\ref{eq1.7}) and
(\ref{eq3.3}) there is $u\in D_e(\EE)$ and a subsequence (still
denoted by $n$) such that $u_n\rightarrow u$ $m$-a.e. and weakly in $V$.
Then, by the definition of $f_n$, $f_n(\cdot,u_n)\rightarrow
f(\cdot,u)$ $m$-a.e. By (\ref{eq3.3}),  for any Borel subset $B$ of $E$  and $a>0$ we have
\begin{align*}
\int_B|f(x,u_n(x))|\,m(dx)&\le a^{-1}\int_{B\cap\{|u_n|>a\}}|f(x,u_n(x))u_n(x)|\,m(dx)\\
&\quad+\int_{B\cap\{|u_n|\le a\}}|f(x,u_n(x))|\,m(dx)\\
&\le a^{-1}\|\mu\|^2_{V'}+\int_B F_a(x)\,m(dx).
\end{align*}
From the above inequality and  (\ref{eq1.5}) we conclude that the sequence $\{f_n(\cdot,u_n)\}$ is equi-integrable and tight.
Hence  $f_n(\cdot,u_n)\rightarrow f(\cdot,u)$ in $L^1(E;m)$ by Vitali's
convergence theorem (see, e.g., \cite[Theorem 2.24]{FL}).
Therefore letting $n\rightarrow\infty$ in (\ref{eq3.1}) we see
that
\begin{equation}
\label{eq3.4}
\EE(u,v)=\int_Ef(x,u(x))v(x)\,m(dx)+\langle\langle\mu,v\rangle\rangle
\end{equation}
for any bounded $v\in D_e(\EE)$. Let $\eta\in L^1(E;m)$ be such that
$\|G|\eta|\|_\infty<\infty$, and let $\{F_n\}$
be a generalized nest such that $\eta_n=\mathbf{1}_{F_n}\eta\in V'$. Then $G\eta_n$ is bounded and $G\eta_n\in V$. Therefore  taking $v=G\eta_n$ as a test function in (\ref{eq3.4}) we get
\[
\int_E u\eta_n\,dm=\int_E f(x,u(x))G\eta_n(x)\,m(dx)
+\int_E \widetilde{G\eta_n}(x)\,\mu(dx).
\]
By Lemma \ref{lem2.1},
$\widetilde{G\eta_n}\rightarrow\widetilde{G\eta}$ as
$n\rightarrow\infty$. Therefore letting $n\rightarrow \infty$ in
the above equation we obtain (\ref{eq2.3}). Thus $u$ is a solution
of
(\ref{eq1.1}).\\
{\em Step 2.} We now show how to dispense with the assumption that
$\mu^{+}$, $\mu^{-}\in\MM_{0,b}\cap V'$. By the 0-order version of
\cite[Theorem 2.2.4]{FOT} (see the beginning  of the proof of
\cite[Theorem 2.4.2(ii)]{FOT}), there exists a generalized nest
$\{F_n\}$ such that $\mu^{(+)}_n=\fch_{F_n}\cdot\mu^+$,
$\mu^{(-)}_n=\fch_{F_n}\cdot\mu^-\in\MM_{0,b}\cap V'$.
Set $\mu_n=\mu^{(+)}_n-\mu^{(-)}_n$.
By {\em Step 1}, there exists a  solution $u_n\in D_e(\EE)$ of the
equation
\begin{equation}
\label{eq3.6} -Lu_n=f(\cdot, u_n)+\mu_n.
\end{equation}
In particular,
\begin{equation}
\label{eq3.13}
\int_E u_n(x)\eta(x)\,m(dx)=\int_E f(x,u_n(x))G\eta(x)\,m(dx)
+\int_E \widetilde{G\eta}(x)\,\mu(dx)
\end{equation}
for any $\eta\in L^1(E;m)$ such that $G|\eta|$ is bounded. By Lemma \ref{lem3.1},
for any Borel subset $B$ of $E$ and $a>0$ we have
\begin{align}
\label{eq3.9}
\int_{B}|f(x,u_n(x))|\,m(dx)&=
\int_{B\cap\{|u_n|\le a\}}|f(x,u_n(x))|\,m(dx)\nonumber\\
&\quad+\int_{B\cap\{|u_n|>a\}}|f(x,u_n(x))|\,m(dx)\nonumber\\
&\le \int_BF_a(x)\,m(dx)+\|\fch_{\{|u_n|>a\}}\cdot\mu\|_{TV}.
\end{align}
By Proposition \ref{prop3.3},
\begin{equation}
\label{eq3.12}
\EE(T_k(u_n),T_k(u_n))\le 2k\|\mu_n\|_{TV},
\end{equation}
whereas  by the $0$-order version of \cite[(2.1.10)]{FOT} and (\ref{eq2.2}),
\[
\mbox{Cap}_{(0)}(\{|T_k(u_n)|>a\})\le a^{-2}\EE(T_k(u_n),T_k(u_n)).
\]
If $k>a$, then $\{|u_n|>a\}=\{|T_k(u_n)|>a\}$, so for any $k>a$,
\begin{align*}
\mbox{Cap}_{(0)}(\{|u_n|>a\})&=\mbox{Cap}_{(0)}(\{|T_k(u_n)|>a\}) \\
&\le a^{-2}\EE(T_k(u_n),T_k(u_n))\le 2a^{-2}k\|\mu_n\|_{TV}.
\end{align*}
Hence
\begin{equation}
\label{eq3.11} \mbox{Cap}_{(0)}(\{|u_n|>a\})\le 2a^{-1}\|\mu\|_{TV}.
\end{equation}
Let $\varepsilon>0$. As $|\mu|\in\MM_{0,b}$, by  Lemma
\ref{lem2.5} there exists $\delta>0$ such that
$|\mu|(\{|u_n|>a\})\le\varepsilon/2$ if
$\mbox{Cap}_{(0)}(\{|u_n|>a\})\le\delta$. Hence, by
(\ref{eq3.11}), $|\mu|(\{|u_n|>a\})\le\varepsilon/2$ if
$a=\delta^{-1}\|\mu\|_{TV}$. By (\ref{eq1.5}) with
$a=\delta^{-1}\|\mu\|_{TV}$, there is $\gamma>0$ such that
$\int_BF_a(x)\,m(dx)<\varepsilon/2$ if $m(B)\le\gamma$. From this and
(\ref{eq3.9}) it follows that if $m(B)\le\gamma$, then
$\int_{B}|f(x,u_n(x))|\,m(dx)\le \varepsilon$. Furthermore, by
(\ref{eq1.5}) and $\sigma$-finitness of $m$, there exists a Borel
set $E_0\subset E$ such that $m(E_0)<\infty$ and $\int_{E\setminus
E_0} F_a(x)\,m(dx)<\varepsilon/2$. Therefore taking $B=E\setminus E_0$
in (\ref{eq3.9}) we get $\int_{E\setminus
E_0}|f(x,u_n(x))|\,m(dx)\le\varepsilon$. This shows that the
sequence $\{f(\cdot,u_n)\}$ is equi-integrable and tight. On the
other hand, by (\ref{eq1.7}) and  (\ref{eq3.12}), for each $k>0$
the sequence $\{T_k(u_n)\}_n$ is, up to a subsequence, convergent
$m$-a.e., so using the diagonal argument, one can find a
subsequence, still denoted by $(n)$, such that $\{u_n\}$ converges
$m$-a.e. to some $u$. Hence $f(\cdot,u_n)\rightarrow f(\cdot,u)$
$m$-a.e. Consequently, by Vitali's theorem,
$f(\cdot,u_n)\rightarrow f(\cdot,u)$ in $L^1(E;m)$. Let $k>0$, and
let $g$ be a strictly positive function such that
$\|Gg\|_{\infty}<\infty$ and $g\in L^1(E;m)$. Taking
$\eta=g\mathbf{1}_{\{u_n\ge k\}}$ and
$\eta=-g\mathbf{1}_{\{u_n\le-k\}}$ as test functions in
(\ref{eq3.13}) we obtain
\begin{align}
\label{eq3.22}
\int_{\{|u_n|\ge k\}} |u_n(x)|g(x)\,m(dx)&
\le -\int_E |f_n(x,u_n(x))| G(g\mathbf{1}_{\{|u_n|\ge k\}})(x)\,m(dx)
\nonumber \\
&\quad+\int_E \widetilde{G(g\mathbf{1}_{\{|u_n|\ge k\}})}(x)\,\mu(dx).
\end{align}
We already know that the sequence $\{f(\cdot,u_n)\}$ is equi-integrable and tight. Furthermore, the functions
$\widetilde{G(g\mathbf{1}_{\{|u_n|\ge k\}})}$ is bounded q.e., and by Lemma \ref{lem2.1}, $\widetilde{G(g\mathbf{1}_{\{|u_n|\ge k\}})}\searrow 0$ q.e. as $k\rightarrow\infty$. Therefore from (\ref{eq3.22}) it follows that the sequence $\{u_n\}$ is equi-integrable with respect to the finite measure $\nu=g\cdot m$. Since $u_n$ is a solution of (\ref{eq3.6}),   for any  $\eta\in L^1(E;m)$ such that $\|G|\eta|\|_{\infty}<\infty$ and any $k\ge 0$ we have
\begin{align}
\label{eq2.3abc}
\int_E u_n(x)\eta(x) g_k(x)\,m(dx)&=\int_E f_n(x,u_n(x))G(\eta g_k)(x)\,m(dx)\nonumber\\
&\quad+\int_E\mathbf{1}_{F_n}(x) \widetilde{G(\eta g_k)}(x)\,\mu(dx),
\end{align}
where $g_k=\frac{kg}{1+kg}$.
By what has already been proved, letting $n\rightarrow\infty$
in (\ref{eq2.3abc}) yields
\begin{align*}
\int_E u(x)\eta(x) g_k(x)\,m(dx)&=\int_E f(x,u(x))G(\eta g_k)(x)\,m(dx)\\
&\quad+\int_E\mathbf{1}_{F}(x) \widetilde{G(\eta g_k)}(x)\,\mu(dx)
\end{align*}
with $F=\bigcup^{\infty}_{n=1}F_n$. Since $\mu(E\setminus F)=0$ by
(\ref{eq2.8}), letting $k\rightarrow\infty$ and using Lemma
\ref{lem2.1} shows that $u$ is a solution to (\ref{eq1.1}).
\end{proof}

\begin{remark}
(i) Let $\bar\mu=f(\cdot,0)\cdot m+\mu$, $\bar
f(x,y)=f(x,y)-f(x,0)$. Then $\bar\mu\in\MM_{0,b}$ and if  $f$
satisfies (\ref{eq1.5}) and (\ref{eq1.6}),   then $\bar f$
satisfies (\ref{eq1.2}) and (\ref{eq1.5}).
Furthermore, $u$ is a solution of the problem $-Lu=\bar
f(x,u)+\bar\mu$ if and only if it is a solution of (\ref{eq1.1}).
Therefore under (\ref{eq1.5}) and (\ref{eq1.6}) there exists a
solution of (\ref{eq1.1}).
\smallskip\\
(ii) If (\ref{eq1.5}) and  (\ref{eq1.6}) are satisfied,   then
{\em Step 2} of the proof of Theorem \ref{th3.5} can be shortened.
Indeed, by (\ref{eq1.6}) and Lemma \ref{lem3.1}(ii),
\[
\|f(\cdot,u_n)-f(\cdot,u_k)\|_{L^1(E;m)}\le \|\mu_n-\mu_k\|_{TV}.
\]
Since
\[
\|\mu_n-\mu\|_{TV}\le \|\mu^{+}-\fch_{F_n}\cdot\mu^{+}\|_{TV}
+\|\mu^{-}-\fch_{F_n}\cdot\mu^{-}\|_{TV}=\mu^{+}(E\setminus F_n)
+\mu^{-}(E\setminus F_n),
\]
we have
\[
\limsup_{n\rightarrow\infty}\|\mu_n-\mu\|_{TV}\le
\mu^{+}(E\setminus F) +\mu^{-}(E\setminus F)=0.
\]
By the above,  $\{f(\cdot,u_n)\}$ is convergent in $L^1(E;m)$. The
rest of the proof runs as the proof of Theorem \ref{th3.5} (see
the reasoning following the statement that  $\{f(\cdot,u_n)\}$ is equi-integrable).
\end{remark}

If (\ref{eq1.10}) is satisfied, then the following Poincar\'e-type
inequality holds true: there exists $c>0$ such that
\begin{equation}
\label{eq2.2} \|u\|_{L^2(E;m)}\le c\EE(u,u)^{1/2},\quad u\in
D(\EE)
\end{equation}
(see \cite[Corollary 2.5]{JS}). Hence, under
(\ref{eq1.10}), $D_e(\EE)=D(\EE)$ and the norms determined  by
$\EE$ and $\EE_1$ are equivalent. It follows in particular that
under the assumptions of  Theorem \ref{th3.5} the solutions of (\ref{eq1.3}) belong to $D(\EE)$.

In general, solutions of (\ref{eq1.3}) are not even locally
integrable (see \cite[Example 5.7]{KR:JFA}). Below we shall see
that a simple condition guaranteeing their integrability is
\[
\|G1\|_{\infty}<\infty.
\]
This condition is sometimes expressed by saying that $E$ is Green-bounded
(see, e.g., \cite{BoB,CZ}; note that in the case where problem
(\ref{eq1.3}) (resp. (\ref{eq1.4})) is considered, $G$ is the Green function
for $\Delta$ (resp. $\Delta^{\alpha/2}$) on $D$. The
Green-bounded domain need not be bounded. For instance, if $L=\Delta$,
then the infinite strip $\{(x,y)\in\BR^2:|x|<a\}$  ($a>0$) in $\BR^2$
is Green-bounded  (see \cite[p. 39]{CZ}).

\begin{lemma}
\label{lem3.3} If \mbox{\rm(\ref{eq2.2})} is satisfied, then $E$
is Green-bounded.
\end{lemma}
\begin{proof}
For the constant $c$ from (\ref{eq2.2}) we set
$\EE^c(u,v)=\EE(u,v)-\frac1{2c^2}(u,v)$, $u,v\in D(\EE)$. Then
$(\EE^c,D(\EE))$ is a regular symmetric Dirichlet form on
$L^2(E;m)$. Obviously,
\begin{equation}
\label{eq.ccc}
\EE(u,v)=\EE^c(u,v)+\frac{1}{2c^2}(u,v),\quad u,v\in D(\EE).
\end{equation}
Let $(G^c_\alpha)_{\alpha>0}$ denote  the resolvent associated
with $(\EE^c,D(\EE))$. From (\ref{eq.ccc}) it follows that
$G=G^c_{(2c^2)^{-1}}$. Hence $ G1=G^c_{(2c^2)^{-1}}1\le 2c^2$
since $(\alpha G^c_\alpha)_{\alpha>0}$ is Markovian.
\end{proof}

\begin{proposition}
\label{prop3.13}
Assume that  $E$  is Green-bounded.
If $u$ is a solution to \mbox{\rm(\ref{eq1.1})}, then $u\in L^1(E;m)$.
\end{proposition}
\begin{proof}
To see this it is enough to consider an increasing sequence of compact sets $\{F_n\}$ such that $\bigcup_{n=1}^{\infty}F_n=E$, take  $\eta_n=\fch_{F_n}\mbox{sign}(u)$  as  test functions in  (\ref{eq2.3}), and use (\ref{eq3.33}) and Fatou's lemma.
\end{proof}



\section{Applications}

In this section we provide some examples of  local and  nonlocal symmetric  transient regular Dirichlet forms satisfying  condition (\ref{eq1.7}). Before proceeding, we make some general comments on conditions (\ref{eq1.10}) and (\ref{eq1.8}).

Since (\ref{eq1.10}) implies (\ref{eq2.2}), it is clear that  (\ref{eq1.10}) implies (\ref{eq1.7}). That the absolute continuity condition  (\ref{eq1.8}) (or, equivalently, condition (\ref{eq1.9})) implies (\ref{eq1.7}) follows from \cite[Propositions 2.4 and 2.11]{K:PA}. We include a direct proof of this fact  for completeness of exposition.

\begin{proposition}
Condition \mbox{\rm(\ref{eq1.8})} implies \mbox{\rm(\ref{eq1.7})}.
\end{proposition}}
\begin{proof}
Assume that $\{u_n\}\subset D_e(\EE)$ and $\sup_{n\ge1}\EE(u_n,u_n)<\infty$.
Choose $v\in D(\EE)$ such that $\|v\|_{\infty}<\infty$ and $v>0$ $m$-a.e., and for $k>0$ set $w^k_n=v\cdot T_ku_n$.
By \cite[Corollary 1.5.1]{FOT}, $w^k_n\in D_e(\EE)$ and
\begin{equation}
\label{eq3.29}
\EE(w^k_n,w^k_n)\le\|v\|_{\infty}\EE(u_n,u_n)
+k\EE(v,v).
\end{equation}
Clearly
\[
(w^k_n-\alpha G_{\alpha}w^k_n,w^k_n-\alpha G_{\alpha}w^k_n)
=(w^k_n,w^k_n-\alpha G_{\alpha}w^k_n) +(\alpha G_{\alpha}w^k_n,\alpha G_{\alpha}w^k_n-w^k_n).
\]
By \cite[Lemma 1.3.4]{FOT}, $\alpha(w^k_n,w^k_n-\alpha G_{\alpha}w^k_n)\le\EE(w^k_n,w^k_n)$ for every $\alpha>0$. Moreover,
\begin{align*}
(\alpha G_{\alpha}w^k_n,\alpha G_{\alpha}w^k_n-w^k_n)
&=\alpha(G_{\alpha}w^k_n,\alpha G_{\alpha}w^k_n) -\EE_{\alpha}(G_{\alpha}w^k_n,\alpha G_{\alpha}w^k_n)\\
&=-\EE(G_{\alpha}w^k_n,\alpha G_{\alpha}w^k_n)\le0.
\end{align*}
By the above estimates,
$\|\alpha G_{\alpha}w^k_n-w^k_n\|^2_{L^2(E;m)}\le\alpha^{-1}\EE(w^k_n,w^k_n)$ for $\alpha>0$,
which when combined with (\ref{eq3.29}) shows that there is a constant $c(k,v)$ depending only on $k$ and $v$ such that
\begin{equation}
\label{eq2.10} \|\alpha G_{\alpha}
w_n^k-w_n^k\|^2_{L^2(E;m)}\le\alpha^{-1}c(k,v).
\end{equation}
Since $\alpha G_{\alpha}1\le1$, from (\ref{eq2.6}) it follows that $\int_Er_{\alpha}(x,y)\,m(dy)\le\alpha^{-1}$. Hence $r_{\alpha}(x,\cdot)\in L^1(E;m)$ for every $x\in E$. Furthermore,
since $\sup_{n\ge1}\|w^k_n\|_{\infty}\le k\|v\|_{\infty}<\infty$, there is
a subsequence $(n')\subset(n)$ such that $\{w^k_{n'}\}$ converges
weakly${}^*$ in $L^{\infty}(E;m)$ to some $w\in L^\infty(E;m)$, i.e. $\int_Ew^k_{n'}(x)\eta(x)\,m(dx)\rightarrow\int_Ew(x)\eta(x)\,m(dx)$ for
every $\eta\in L^1(E;m)$.
In particular, for every $x\in E$,
\begin{align}
\label{eq3.32}
G_\alpha w^k_{n'}(x)&=\int_Er_\alpha(x,y)w^k_{n'}(y)\,m(dy)\nonumber\\
&\quad\rightarrow
\int_Er_\alpha(x,y)w(y)\,m(dy)=G_\alpha w(x).
\end{align}
Since $|\alpha G_{\alpha}w^k_{n}|\le k\alpha G_\alpha |v|$, it follows from (\ref{eq3.32}) that the sequence$\{\alpha G_\alpha w^k_{n'}\}$ converges
 in $L^2(E;m)$ for any fixed $\alpha>0$, $k>0$. This and
(\ref{eq2.10}) imply that  there exists a subsequence
$(n'')\subset(n')$ such that $\{w^k_{n''}\}$ converges in
$L^2(E;m)$. Using the diagonal procedure one can  find a further  subsequence
$(n''')\subset(n'')$ such that   $\{u_{n'''}\}$ converges $m$-a.e. on $E$.
\end{proof}

\begin{example}
\label{ex3.4}  Let $D\subset\BR^d$, $d\ge1$, be a nonempty  bounded open
set, and let $a_{ij}:D\rightarrow\BR$ be locally integrable
functions such that $a_{ij}(x)=a_{ji}(x)$ for $x\in D$,
$i,j=1,\dots,d$, and for some $\lambda>0$,
\[
\sum^d_{i,j=1}a_{ij}(x)\xi_i\xi_j\ge\lambda|\xi|^2,\quad
x\in D,\,\xi=(\xi_1,\dots,\xi_d)\in\BR^d.
\]
(i) (Dirichlet boundary conditions) The form defined by
\begin{equation}
\label{eq3.16} \EE(u,v)=\sum^d_{i,j=1}\int_D\frac{\partial
u}{\partial x_i}(x)\frac{\partial v}{\partial x_j}(x) a_{ij}(x)\,dx,\quad u,v\in D(\EE)
\end{equation}
with $D(\EE)=H^1_0(D)$ is a regular symmetric Dirichlet form on $L^2(D)$  (see, e.g., \cite[Section 3.1]{FOT}). The generator $L$ of $(\EE,D(\EE))$ is of the form
\[
Lu=\sum^d_{i,j=1}\frac{\partial}{\partial x_j}\Big(a_{ij}\frac{\partial u }{\partial x_i}\Big),\quad u\in D(L).
\]
The form $(\EE,D(\EE))$  is transient by    Poincar\'e's
inequality, and (\ref{eq1.10}) is satisfied by Rellich's theorem.
Also note that by classical results (see \cite{A}),
assumption (\ref{eq1.9}) (and hence (\ref{eq1.8})) is satisfied as well.
Therefore Theorem \ref{th3.5} applies to the  Dirichlet problem
\begin{equation}
\label{eq3.25} -Lu=f(\cdot, u)+\mu\quad\mbox{in }D, \qquad
u=0\quad\mbox{on }\partial D.
\end{equation}
Note that  we impose no regularity  assumption  on the boundary
$\partial D$ of $D$.  Note also that by Poincar\'e's
inequality, $D_e(\EE)=H^1_0(D)$. Consequently,
$T_k(u)\in H^1_0(D)$ for every $k>0$. Furthermore,  since
$D$ is Green-bounded (see, e.g., \cite[Theorem 1.17]{CZ}),
$u\in L^1(D;dx)$ by Proposition \ref{prop3.13}.
\smallskip\\
(ii) (Neumann boundary conditions) Assume additionally  that
$\partial D$ is Lipschitz. Consider the form $\EE$ defined by
(\ref{eq3.16}), but with domain $H^1(D)$. Then $(\EE,H^1(D))$ is a
regular symmetric Dirichlet form on $L^2(\bar D;dx)$ with $\bar
D=D\cup\partial D$ (see \cite[Example 4.5.3]{FOT}), and clearly so
is $(\EE_{\lambda},H^1(D))$ with $\lambda\ge0$. Moreover, if
$\lambda>0$, then  $(\EE_{\lambda},H^1(D))$ is transient because
$D$ is Green-bounded (see Lemma \ref{lem3.3}). The generator
$L^{\lambda}$ of $(\EE_{\lambda},H^1(D))$ is equal to $L-\lambda$,
where $L$ is the generator of $(\EE,H^1(D))$. By
Rellich's theorem, $H^1(D)\hookrightarrow L^2(D;dx)$ is compact,
so the results of the paper apply to equation (\ref{eq1.1}) with
$L$ replaced by the operator $L-\lambda$ defined above. A solution
$u$ to such equation can be viewed as a solution to the Neumann
problem
\[
-Lu=-\lambda u+f(\cdot, u)+\mu\quad\mbox{in }D,\qquad \frac{\partial u}{\partial{(a\cdot\mathbf{n}})}=0\quad\mbox{on }\partial D,
\]
where  ${\mathbf{n}}$ denotes the unit  outward normal to
$\partial D$.
\end{example}

\begin{example}
\label{ex3.5}
Assume that $f$ satisfies (\ref{eq1.2}) and (\ref{eq1.5}). Let  $\psi:\BR^d\rightarrow[0,\infty)$ be a  continuous negative
definite function in the sense of Schoenberg (see \cite[Chapter
3]{J1} for the definition). Denote by $H^{\psi,1}(\BR^d)$ the
space
\[
H^{\psi,1}(\BR^d)=\{u\in L^2(\BR^d):\|u\|_{\psi,1}<\infty\},
\]
where
\[
\|u\|^2_{\psi,1}=\int_{\BR^d}(1+\psi(x))|\hat u(x)|^2\,dx
\]
and $\hat u$ stands for the Fourier transform of $u$.
It is  known (see \cite[Example 1.4.1]{FOT} or  \cite[Example
4.1.28]{J1}) that $(\EE,D(\EE))$ defined as
\begin{equation}
\label{eq3.27} \EE(u,v)=\int_{\BR^d}\hat u(x)\overline{\hat
v(x)}\psi(x)\,dx,  \quad u,v\in D(\EE):=H^{\psi,1}(\BR^d)
\end{equation}
is a symmetric regular Dirichlet form on $L^2(\BR^d;dx)$.  By
\cite[Example 1.5.2]{FOT}, it is transient if and only if
\begin{equation}
\label{eq3.28}
\frac{1}{\psi}\in L^1_{loc}(\BR^d).
\end{equation}
(i) Let $\nu_t$, $t>0$, be a probability measure on $\BR^d$ such that $\hat\nu_t(x)=e^{-t\psi(x)}$, $x\in\BR^d$. Then the semigroup $(P_t)_{t>0}$ associated with $\EE$ has the form $P_tf(x)=\int_{\BR^d}f(x+y)\,\nu_t(dy)$ for $f\in L^2(\BR^d;dx)\cap\BB_b(E)$ (see \cite[Example 1.4.1]{FOT}).
It follows in particular that if $\nu_t$ are absolutely continuous with respect to the Lebesgue measure, then (\ref{eq1.9}) is satisfied. For instance, this is the case when $\psi(\xi)=|\xi|^{\alpha}$, $\xi\in\BR^d$, with $\alpha\in(0,2]$ (see \cite[Example 1.4.1]{FOT}). For such $\psi$, the operator corresponding to $\EE$ is the fractional Laplacian $\Delta^{\alpha/2}$. If $\alpha<d$, then (\ref{eq3.28}) is satisfied, so the form $\EE$ is transient. Therefore Theorem \ref{th3.5} applies to the equation
\[
-\Delta^{\alpha/2}=f(\cdot,u)+\mu\quad\mbox{in }\BR^d
\]
with $\alpha\in(0,2\wedge d)$. If $f$ satisfies (\ref{eq1.2}) and $u$ is a solution to the above equation, then $T_k(u)\in D_e(\BR^d)$ for any $k>0$. For the characterisation of $D_e(\EE)$ see \cite[Example 1.5.2]{FOT}.
Finally, let us note that some general conditions ensuring (\ref{eq1.9}) are found in \cite[Section 27]{S}.
\smallskip\\
(ii) Let $D\subset\BR^d$ be a nonempty bounded open set,  and let
$(\EE^D,D(\EE^D))$ denote the part of $(\EE,D(\EE))$ on $D$, i.e.,
\begin{equation}
\begin{cases}
\label{eq3.15} D(\EE^D)=\{u\in D(\EE):\tilde u=0\mbox{ q.e. on
}\BR^d\setminus D\}, \smallskip\\ \EE^D(u,v)=\EE(u,v),\quad u,v\in D(\EE^D)
\end{cases}
\end{equation}
(here $\tilde u$ denotes a quasi-continuous version of $u$). By
\cite[Theorem 4.4.3]{FOT}, $(\EE^D,D(\EE^D))$ is a  symmetric
regular Dirichlet form on $L^2(D;dx)$ and
$D(\EE^D)=H^{\psi,1}_0(D)$, where $H^{\psi,1}_0(D)$ denotes the
closure of $C^{\infty}_c(D)$ in $H^{\psi,1}(\BR^d)$. If
(\ref{eq3.28}) is satisfied, then the form $(\EE^D,D(\EE^D))$ is
transient by \cite[Theorem 4.4.4]{FOT}.
Let $L$ denote the generator of $(\EE,D(\EE))$ and $L^D$  denote
the generator of $(\EE^D,D(\EE^D))$. By virtue of (\ref{eq3.15}),
the solution $u$ of (\ref{eq1.1}) with $E=D$ and operator $L^D$ can be
interpreted as a solution of the  Dirichlet problem
\begin{equation}
\label{eq3.31}
-Lu=f(\cdot, u)+\mu\quad\mbox{in }D,\qquad u=0\quad\mbox{on }\BR^d\setminus D.
\end{equation}
By \cite[Remark 3.10.6]{J1}, the embedding of
$V:=H^{\psi,1}_0(D)$ (equipped with the norm $\|\cdot\|_{\psi,1}$)
into $L^2(D;dx)$ is compact if and only if
\begin{equation}
\label{eq3.30}
\lim_{|\xi|\rightarrow\infty}\psi(\xi)=\infty.
\end{equation}
Therefore, if (\ref{eq3.28}) and (\ref{eq3.30}) are  satisfied,
then by Theorem \ref{th3.5} there exists a solution $u$ to
(\ref{eq3.31}). Since (\ref{eq1.10}) implies (\ref{eq2.2}), $D_e(\EE^D)=D(\EE^D)$. Consequently,
$T_k(u)\in H^{\psi,1}_0(D)$ for $k>0$.

For instance, (\ref{eq3.28}) and (\ref{eq3.30}) are satisfied for $\psi$ defined as
$\psi(\xi)=|\xi|^{\alpha}$, $\xi\in\BR^d$, with $\alpha\in(0,2\wedge d)$. Since then $L=\Delta^{\alpha/2}$, equation (\ref{eq1.1}) with
$L^D$ can be interpreted as (\ref{eq1.4}). Note also that, because
$D$ is bounded, it is Green-bounded (see, e.g., \cite[(2.4)]{BoB}).
Therefore, by Proposition \ref{prop3.13}, if $u$ is a solution to
(\ref{eq1.4}), then $u\in L^1(D;dx)$.
Other examples of $\psi$ satisfying (\ref{eq3.28}) and
(\ref{eq3.30}) are found in  \cite[Chapter 3]{J1}.
\end{example}



\end{document}